\documentclass[12pt]{amsart}
\usepackage[utf8]{inputenc}
\usepackage{amssymb, amsmath}
\usepackage{fourier}
\usepackage{enumerate}

\usepackage[a4paper, left=30mm,right=30mm,top=30mm,bottom=30mm]{geometry}

\theoremstyle{plain}

\newtheorem{theorem}{Theorem}

\newtheorem{lemma}[theorem]{Lemma}

\theoremstyle{remark}

\newtheorem{remark}[theorem]{Remark}

\newcommand{\bit}{\begin{itemize}}
\newcommand{\eit}{\end{itemize}}
\newcommand{\ben}{\begin{enumerate}}
\newcommand{\een}{\end{enumerate}}
\newcommand{\be}{\begin{equation}}
\newcommand{\ee}{\end{equation}}
\newcommand{\ba}{\begin{array}}
\newcommand{\ea}{\end{array}}

\newcommand{\supp}[1]{\mathrm{supp}\left(#1\right)}

\newcommand{\abs}[1]{\left|#1\right|}
\newcommand{\norm}[1]{\left|\left|#1\right|\right|}

\newcommand{\s}{\mathrm{s}}

\newcommand\bx{\mathbf{x}}
\newcommand\bn{\mathbf{n}}

\newcommand\bnull{\mathbf{0}}

\newcommand{\ol}[1]{\overline{#1}}

\newcommand\cC{\mathcal C}

\newcommand\cH{\mathcal H}

\newcommand\cP{\mathcal P}
\newcommand\cS{\mathcal S}
\newcommand\C{\mathbb C}
\newcommand\R{\mathbb R}

\newcommand\Z{\mathbb Z}

\renewcommand\l{\lambda}
\newcommand{\dl}{\mathrm{d}\l}
\newcommand{\dmu}{\mathrm{d}\mu}

\newcommand{\ler}[1]{\left( #1 \right)}

\newcommand\T{\mathbb T}

\title{Applications of an intersection formula to dual cones}

\author{D\'aniel Virosztek}
\address{Department of Analysis, Institute of Mathematics\\
Budapest University of Technology and Economics\\
H-1521 Budapest, Hungary
and
MTA-DE ``Lend\" ulet'' Functional Analysis Research Group, Institute of Mathematics\\
         University of Debrecen\\
         H-4002 Debrecen, P.O. Box 400, Hungary}
\email{virosz@math.bme.hu}
\urladdr{http://www.math.bme.hu/\~{}virosz}


\thanks{
The author was supported by the ``Lend\" ulet'' Program (LP2012-46/2012) of the Hungarian Academy of Sciences and by the National Research, Development and Innovation Office -- NKFIH, Grant No. K104206.}

\begin{document}
\begin{abstract}
We give a succinct proof of a duality theorem obtained by R\'ev\'esz in $1991$ \cite{revesz-aus} which concerns extremal quantities related to trigonomertic polynomials.
The key tool of our new proof is an intersection formula on dual cones in real Banach spaces. We show another application of this intersection formula which is related to the integral estimates of non-negative positive definite functions.
\end{abstract}
\maketitle

\section{Introduction}

Let $X$ be a real Banach space and let $X'$ denote its topological dual space endowed with the weak-$^*$ topology. For any set $D \subseteq X,$ the \emph{dual cone} of $D$ is denoted by $D^+$ and is defined as 

$$D^+=\left\{ \varphi \in X' \middle| \varphi(x) \geq 0 \quad \forall x \in D\right\},$$
see, e.g., \cite[Section 2]{jeya-wolk}. The \emph{polar cone} (denoted by $D^-$) is defined as $D^-:=-D^+.$ Note that both $D^+$ and $D^-$ are weak-$^*$ closed convex cones in $X',$ no matter what the set $D$ is.
Moreover, by \cite[Lemma 2.1.]{jeya-wolk}, if $\cC$ and $\cP$ are convex sets in $X$ such that $0 \in \cC \cap \cP$ and $\cC \cap \mathrm{int} \, \cP\neq \emptyset$, then
\be \label{metszet}
\ler{\cC \cap \cP}^+ =\cC^+ + \cP^+.
\ee
Consequently, in this case we have $\ler{\cC \cap \cP}^- =\cC^- + \cP^-.$
\par
In this short note we show two applications of the formula \eqref{metszet} that describes the structure of the dual cone of the intersection of cones. Both applications are of a Fourier-analytic nature, hence we collect some basic facts and notation of this topic below.

\ben [(i)]
\item 
For a locally compact abelian group $G,$ the symbol $M(G)$ denotes the set of all complex-valued regular Borel measures on $G$ with finite total variation. $M(G)$ is a commutative, unital Banach algebra, where the norm is defined as $\norm{\mu}=\abs{\mu}(G),$ and the multiplication is defined by the convolution \cite[1.3.2. Corollary]{rudin}.

\item
The symbol $L^1(G)$ stands for the set of all integrable functions on $G$ (with respect to the Haar measure, which is denoted by $\lambda.$) We may consider $L^1(G)$ as a subset of $M(G)$ by the embedding
$$
\mu_{(\cdot)}: L^1(G) \rightarrow M(G), \, f \mapsto \mu_f; \, \mu_f(E)=\int_E f \dl \text{ for any Borel set } E \subseteq G.
$$
In fact $L^1(G)$ is a Banach subalgebra of $M(G)$ \cite[1.3.5. Theorem]{rudin}. Moreover, $L^1(G)$ is unital if and only if $L^1(G)=M(G)$ if and only if $G$ is discrete \cite[1.7.3. Theorem]{rudin}.

\item $L^\infty (G)$ stands for the set of all essentially bounded measurable functions on $G.$

\item 
A continuous group homomorphism from the locally compact abelian group $G$ into the multiplicative group $\mathbb{T}=\{z \in C | \abs{z}=1\}$ is called a \emph{character} of $G.$ The set of all characters of $G$ forms a group (with pointwise multiplication) which is called the \emph{dual group} of $G,$ and it is denoted by $\hat G.$
\item \label{fubini}
For any $\mu \in M(G)$ (or $f \in L^1(G)$), the symbol $\hat \mu$ (or $\hat f$) denotes the Fourier transform of $\mu$ (or $f$), that is,
$$
\hat \mu \in \C^{\hat G}; \quad \hat \mu (\gamma)= \int_G \ol \gamma \dmu \quad \ler{\gamma \in \hat G}
$$
and
$$\hat f \in \C^{\hat G}; \quad \hat f (\gamma)= \int_G f \ol \gamma \dl \quad \ler{\gamma \in \hat G}.
$$
The Fourier transform is a continuous linear transformation from $L^1(G)$ into $C_0\ler{\hat G},$ where $C_0\ler{\hat G}$ denotes the set of all functions on $\hat G$ vanishing at infinity (the topology on $\hat G$ is the weak topology induced by the set of all functions $\hat f$ obtained as Fourier-transforms of $L^1$ functions on $G$).
Moreover, it is a contraction as $\norm{\hat f}_{\infty} \leq \norm{f}_1.$ (For details, see \cite[1.2.4. Theorem]{rudin}.)
\par
The following useful formula is an easy consequence of Fubini's theorem. If $\mu \in M(G), \nu \in M\ler{\hat G}$ and $\phi(x)=\int_{\hat G} \gamma(x) \mathrm{d}\nu(\gamma) \, (x \in G),$ then
\be \label{parseval}
\int_G \overline{\phi} \mathrm{d} \mu= \int_{\hat G} \hat \mu \mathrm{d} \ol \nu. 
\ee
\een

Now we turn to the detailed descriptions of the applications of the intersection formula \eqref{metszet}. Section \ref{ujbiz} is devoted to describe the first one, and Section \ref{c-k} contains the second one.


\section{A new proof of a duality theorem}
\label{ujbiz}

In 1991, R\'ev\'esz proved a duality theorem on certain extremal quantities related to multivariable trigonometric polynomials \cite{revesz-aus}. That theorem is general enough to cover the duality statements appearing in \cite{revesz-period}, \cite{revesz-imacs} and \cite{ruzsa}.
The setting of the theorem is as follows.
\par
Let $d$ be a positive integer. Let us use the notation $\T^d=\ler{\R/2 \pi \Z}^d$ and
$$
\Z_+^d=\left\{\bn=(n_1, \dots, n_d) \in \Z^d \middle|
\exists j \in \{1,2, \dots, d\} \text{ such that } n_k=0 \text{ for any } k < j \text{ and } n_{j}>0 \right\}.
$$
Let $M \subseteq \Z_+^d, \, L \subseteq \Z_+^d,$ and let $M^c:=\Z_+^d \setminus M$ and $ L^c:=\Z_+^d \setminus L.$
\par
Consider the real Banach space $L_{\R,\s}^1\ler{\Z^d}$ of all symmetric real-valued absolutely summable functions on $\Z^d$ with its topological dual space $L_{\R,\s}^\infty\ler{\Z^d}.$

Set 
$$
\cC:=
\left\{
f \in L_{\R,\s}^1\ler{\Z^d}
\middle|
f \text{ has finite support, }
\right.
$$
$$
\left.
\supp{f_+} \subseteq \{\bnull\} \cup M \cup -M \text{ and }
\supp{f_-} \subseteq \{\bnull\} \cup L \cup -L
\right\}.
$$
The set $\cC \subseteq  L_{\R,\s}^1\ler{\Z^d}$ is a convex set. It is easy to see that the dual cone of $\cC$ is
$$
\cC^+=\left\{t \in L_{\R,\s}^\infty\ler{\Z^d}\middle| \supp{t_+} \subseteq  L^c \cup -L^c \text{ and }
\supp{t_-} \subseteq M^c \cup -M^c
\right\}.
$$
Therefore, the polar cone of $\cC$ is
$$
\cC^-=\left\{t \in L_{\R,\s}^\infty\ler{\Z^d}\middle| 
\supp{t_+} \subseteq M^c \cup -M^c
\text{ and }
\supp{t_-} \subseteq  L^c \cup -L^c
\right\}.
$$
\par
Set
$$
\cP:=\left\{f \in L_{\R,\s}^1\ler{\Z^d}\middle|
\hat f(\bx)=\sum_{\bn \in \Z^d} f(\bn) e^{-i \bn \cdot \bx} =
f(\bnull)+2 \sum_{\bn \in \Z_+^d} f(\bn) \cos{\ler{\bn \cdot \bx}}
\geq 0 \, \forall \bx \in \T^d\right\}.
$$
Clearly, $\cP$ is a convex set. The following Lemma is devoted to describe its dual cone.

\begin{lemma} \label{elso-lem}
$$
\cP^+=\left\{h \in L_{\R, \s}^\infty\ler{\Z^d} \middle| \,h \gg 0,\text{ that is, } h \text{ is positive definite}\right\}.
$$
\end{lemma}
\begin{proof}[Proof of Lemma \ref{elso-lem}]
Recall that a function $h \in L_{\R, \s}^\infty\ler{\Z^d}$ is said to be positive definite if $\sum_{i,j=1}^m z_i \overline{z_j} g\ler{\mathbf{n_i}-\mathbf{n_j}}\geq 0$ holds for any $\mathbf{n_1}, \dots,\mathbf{n_m} \in \Z^d$ and $z_1, \dots,z_m \in \C.$ However, for symmetric real functions, positive definiteness is equivalent to the a priori weaker condition $\sum_{i,j=1}^m c_i c_j g\ler{\mathbf{n_i}-\mathbf{n_j}}\geq 0 \, \ler{\mathbf{n_1}, \dots,\mathbf{n_m} \in \Z^d, \, c_1, \dots, c_m \in \R}.$
\par
Let us recall Bochner's theorem \cite[1.4.3. Theorem]{rudin} which says that a function $h \in L_{\R, \s}^\infty\ler{\Z^d}$ is positive definite if and only if there is a non-negative symmetric measure $ \nu \in M_{\R, \s}\ler{\T^d}$ such that
$$
h(\bn)=\int_{\T^d} e^{i \bn \cdot \bx} \mathrm{d}\nu(\bx) \qquad \ler{ \bn \in \Z^d}.
$$
Therefore, the positive definite functions are in $\cP^+$ as any positive definite $h \in L_{\R,\s}^\infty\ler{\Z^d}$ can be written in the form $h(\bn)=\int_{\T^d} e^{i \bn \cdot \bx} \mathrm{d}\nu(\bx) \, \, \ler{ \bn \in \Z^d}$ for some $0 \leq \nu \in M_{\R, \s}\ler{\T^d},$ and hence, by equation \eqref{parseval}, the inequality
$$
\sum_{\bn \in \Z^d} f(\bn)  h(\bn)=\int_{\T^d} \hat f(\bx) \mathrm{d} \nu(\bx) \geq 0
$$
for any $f \in \cP \subset L_{\R,\s}^1\ler{\Z^d}.$
\par
Conversely, if $g \in L_{\R,\s}^\infty\ler{\Z^d}$ and $g$ is not positive definite, that is,
$$
\sum_{i,j=1}^m c_i c_j g\ler{\mathbf{n_i}-\mathbf{n_j}}<0
$$
for some $\{\mathbf{n_1}, \dots,\mathbf{n_m}\} \subset \Z^d$ and $\{c_1, \dots,c_m\}\subset \R,$ then
$$
\sum_{\bn \in \Z^d} \ler{x * \tilde x}(\bn) g(\bn)<0,
$$
where $x=\sum_{i=1}^m c_i \chi_{\{\mathbf{n_i}\}}$ and $\tilde x$ is defined by $\tilde x (\mathbf{n})=x(-\mathbf{n})\, \ler{\mathbf{n} \in \Z^d}.$ Clearly, $x * \tilde x \in \cP \subset L_{\R,\s}^1\ler{\Z^d},$ hence this means that $g \notin \cP^+.$
\end{proof}

Now, let $r \in L_{\R,\s}^\infty \ler{\Z^d}$ with $r(\bnull)=1$ be fixed and let us define the affine subspace
$$
\cH:=\left\{f \in L_{\R,\s}^1 \ler{\Z^d} \middle| f(\bnull)=1\right\}.
$$
According to \cite[eq. (5) and eq. (12)]{revesz-aus}, let us define the extremal quantities
$$
\alpha:=\mathrm{inf} \left\{\sum_{\bn \in \Z^d} f(\bn) r(\bn) \middle| f \in \cC \cap \cP \cap \cH \right\}
$$
and
$$
\omega:=\mathrm{sup} \left\{\delta \in \R \, \middle| \exists t \in \cC^- \text{ such that } r+t-\delta \chi_{\{\bnull\}} \in \cP^+ \right\}
$$
$$
=\mathrm{sup} \left\{\delta \in \R \, \middle| \exists t \in \cC^- \text{ such that } \delta \chi_{\{\bnull\}}-r - t \in \cP^- \right\}
$$
$$
=\mathrm{sup} \left\{\delta \in \R \, \middle| \delta \chi_{\{\bnull\}}-r \in \cC^- +  \cP^- \right\}
$$
(It is clear that the definition on $\alpha$ coincides with the definition given in \cite[eq. (5)]{revesz-aus}. It is less obvious that the definition of $\omega$ is the same as the one given in \cite[eq. (12)]{revesz-aus}. However, the fact that the nonnegative symmetric measures on $\T^d$ are in one-to-one correspondence with the real positive definite functions on $\Z^d$ by the Fourier transform may convince the reader that the definition of $\omega$ is also correct.)
\par
We mentioned before that $\cC$ and $\cP$ are convex sets in the real Banach space $L_{\R,\s}^1\ler{\Z^d}$. It is clear that $0 \in \cC \cap \cP$ and $\cC \cap \mathrm{int} \, \cP\neq \emptyset,$ as $\chi_{\{\bnull\}} \in \cC \cap \mathrm{int} \, \cP.$ (The fact that $\chi_{\{\bnull\}} \in \mathrm{int} \, \cP$ can be easily seen by the following. The Fourier transform is a contraction from $L_{\R,\s}^1\ler{\Z^d}$ into $C_{\R,\s}\ler{\T^d},$ and $\widehat{\chi_{\{\bnull\}}}=1.$)
Therefore, by \cite[Lemma 2.1.]{jeya-wolk}, the intersection formula
$$
\ler{\cC \cap \cP}^+ =\cC^+ + \cP^+
$$
holds.
Consequently, we have $\ler{\cC \cap \cP}^- =\cC^- + \cP^-.$
So, by this intersection formula, $\omega$ can be rewritten as
$$
\omega=\mathrm{sup} \left\{\delta \in \R \, \middle| \delta \chi_{\{\bnull\}}-r \in \ler{\cC\cap \cP}^- \right\}.
$$

\begin{theorem}[R\'ev\'esz, \cite{revesz-aus}]
$$\alpha=\omega.$$
\end{theorem}

\begin{proof}[A short proof]
If $\delta \chi_{\{\bnull\}}-r \in \ler{\cP\cap \cC}^-,$ then
$$
\sum_{\bn \in \Z^d} f(\bn) r(\bn) \geq \delta \text{ for any } f \in  \cC \cap \cP \cap \cH,
$$
as in this case
$$
0 \geq \sum_{\bn \in \Z^d} f(\bn) \ler{\delta \chi_{\{\bnull\}}(\bn)-r(\bn)}
=
\delta f(\bnull)-\sum_{\bn \in \Z^d} f(\bn)r(\bn)
=
\delta-\sum_{\bn \in \Z^d} f(\bn)r(\bn).
$$
Therefore, $\omega \leq \alpha.$
\par
On the contrary, if $\beta> \omega,$ then $ \beta \chi_{\{\bnull\}}-r \notin \ler{\cC\cap \cP}^-,$ that is, there exists some $f \in \cC\cap \cP$ such that

$$
\sum_{\bn \in \Z^d} f(\bn) \ler{\beta \chi_{\{\bnull\}}(\bn)-r(\bn)}> 0.
$$

This $f$ is necessarily a non-zero element of $\cP,$ hence $f(\bnull)>0$. Therefore, without loss of generality, we can assume that $f(\bnull)=1,$ so there exists some $f \in \cC \cap \cP \cap \cH$ such that
$$
\sum_{\bn \in \Z^d} f(\bn) \ler{\beta \chi_{\{\bnull\}}(\bn)-r(\bn)} > 0.
$$
That is,
$$
\beta > \sum_{\bn \in \Z^d} f(\bn) r(\bn)
$$
for some $f \in \cC \cap \cP \cap \cH,$ which means that $\beta >\alpha.$
So, we deduced that $\beta>\omega$ implies $\beta>\alpha,$ therefore, $\alpha \leq \omega.$ The proof is done. 
\end{proof}

\section{Another application of the intersection formula}
\label{c-k}

The second application concerns integral estimates of non-negative positive definite functions. This problem is related to Wiener's problem \cite{shapiro,wiener} and to the recent works \cite{gorbacsev,shteinikov}.
The arguments in this section are partially parallel to the arguments presented in the previous section.

Let $L_{\R, \s}^1(\Z)$ denote the real Banach space of all real-valued, symmetric, summable functions on $\Z$ and let us consider its topological dual space $L_{\R, \s}^{\infty}\ler{\Z}$ endowed with the weak-$^*$ topology. 
\par
Let us define $\cC=\left\{ f \in L_{\R,\s}^1(\Z) \middle| f \geq 0\right\}$ and $\cP=\left\{ f \in L_{\R,\s}^1(\Z) \middle| \hat f \geq 0\right\}.$ Clearly, $\cC$ and $\cP$ are convex cones in $L_{\R,s}^1(\Z).$ The closedness of $\cC$ is obvious, and $\cP$ is also closed as the Fourier transform is a continuous (moreover, norm-non-increasing) linear transformation from $L_{\R,\s}^1(\Z)$ into $C_{\R,\s}(\T),$ and the nonnegative functions form a closed set of $C_{\R,\s}(\T)$ with respect to the maximum norm topology. (The symbol $\T$ denotes the additive group of real numbers modulo $2 \pi$ and $C_{\R,\s}(\T)$ stands for the Banach space of all continuous, symmetric real functions on $\T.$) 
\begin{lemma} \label{masodik-lem}
$$\cC^+=\{g \in L_{\R, \s}^\infty(\Z) | \, g \geq 0\},$$
and  
$$
\cP^+=\{h \in L_{\R, \s}^\infty(\Z) | \,h \gg 0,\text{ that is, } h \text{ is positive definite}\}.
$$
\end{lemma}
\begin{proof}[Proof of Lemma \ref{masodik-lem}]
The first statement of Lemma \ref{masodik-lem} is obvious. The proof of the second statement is very similar to the proof of Lemma \ref{elso-lem}.
\end{proof}

Let $L$ and $N$ be positive integers. Let us define the extremal quantities
$$
C(L,N):=\mathrm{inf}\left\{C \in \R \, \middle| \, \sum_{k=-L N}^{L N} f(k) \leq (C+1) \sum_{k=-N}^{N} f(k) \text{ for any } f \in \cC \cap \cP\right\}
$$
and
$$
K(L,N):=\mathrm{inf}\left\{h(0) \, \middle| \, h \in L_{\R,\s}^\infty(\Z), \, h \gg 0 \text{ and } h(k)\leq -\chi_{\{-L N, \dots,L N\}}(k) \text{, if } \abs{k}>N\right\}.
$$
Let us introduce
$$
\cS_{L,N}:=\left\{h \in L_{\R,\s}^\infty(\Z) \, \middle| \, h(k)\leq -\chi_{\{-L N, \dots,L N\}}(k) \text{, if } \abs{k}>N\right\},
$$
and observe that $\cS_{L,N}$ is closed in the weak-$^*$ topology as it is the intersection of weak-$^*$ closed sets.
\par
Note that
$$
C(L,N)=\mathrm{inf}\left\{C \in \R \, \middle| \, (C+1) \chi_{\{-N, \dots, N\}}-\chi_{\{-L N, \dots, L N\}} \in \ler{\cC \cap \cP}^+ \right\},
$$
and by the result of Lemma \ref{masodik-lem},
$$
K(L,N)=\mathrm{inf}\left\{h(0) \, \middle| \, h \in \cP^+ \cap \cS_{L,N}  \right\}.
$$
\begin{remark} \label{finite}
Let us note that $K(L,N)$ is finite as the set $\cP^+ \cap \cS_{L,N}$ is not empty. Indeed, one can easily check that the function
$$
w_{L,N}(k):= \begin{cases}
2(L-1)N & \text{ if } k=0\\
-1 & \text{ if } N<\abs{k}\leq L N \\
0 & \text{ otherwise}
\end{cases}
$$
is positive definite, and therefore, it is an element of $\cP^+ \cap \cS_{L,N}.$
\end{remark}

\begin{theorem} \label{dual2}
$$
C(L,N)=K(L,N).
$$
\end{theorem}

\begin{proof}
The key idea is the observation that $\cC$ and $\cP$ are convex sets in $L_{\R,\s}^1(\Z)$ such that $0 \in \cC \cap \cP$ and $\cC \cap \mathrm{int} \, \cP\neq \emptyset$ as $\chi_{\{0\}} \in \cC \cap \mathrm{int} \, \cP.$ Therefore, the intersection formula
$$\ler{\cC \cap \cP}^+ =\cC^+ + \cP^+$$
holds.
\par
On the one hand, if $h \in \cP^+ \cap \cS_{L,N}$ then
$$
h \leq \ler{h(0)+1} \chi_{\{-N, \dots, N\}}-\chi_{\{-L N, \dots, L N\}}
$$
as $h(0)\geq h(n) \,\ler{n \in \Z}$ holds for any positive definite function $h \in L_{\R,\s}^\infty(\Z).$ Therefore, in this case
$$
\ler{h(0)+1} \chi_{\{-N, \dots, N\}}-\chi_{\{-L N, \dots, L N\}} \in \cC^+ +\cP^+ =\ler{\cC \cap \cP}^+,
$$
hence $C(L,N) \leq K(L,N).$
\par
On the other hand, by the intersection formula,
$$
C(L,N)=\mathrm{inf}\left\{C \in \R \, \middle| \, (C+1) \chi_{\{-N, \dots, N\}}-\chi_{\{-L N, \dots, L N\}} \in \cC^+ +\cP^+\right\}
$$
holds, hence the following argument shows the opposite inequality. If
$$ (C+1) \chi_{\{-N, \dots, N\}}-\chi_{\{-L N, \dots, L N\}} \in \cC^+ +\cP^+$$
then $ (C+1) \chi_{\{-N, \dots, N\}}-\chi_{\{-L N, \dots, L N\}}=g+h$
for some $g \in \cC^+$ and $h \in \cP^+.$ Clearly, this $h$ is an element of $\cP^+ \cap \cS_{L,N},$ and $h(0)\leq C,$ hence $K(L,N) \leq C(L,N).$
\end{proof}

\begin{remark} \label{finite2}
We have noted (see Remark \ref{finite}) that $K(L,N)$ is finite. Therefore, the result of Theorem \ref{dual2} directly implies the finiteness of $C(L,N).$
\end{remark}

\paragraph*{{\bf Acknowledgement}}
The author is grateful to Szil\'ard R\'ev\'esz for proposing the problems discussed in this paper and for useful communication. The author is also grateful to the anonymous referee for his/her comments which helped to improve the presentation of the paper.

\end{document}